\documentclass{article}
\usepackage[makeroom]{cancel}
\usepackage{amssymb,graphicx,graphics,color,amsopn,amsfonts,amsmath,hyperref}
\newtheorem{theorem}{Theorem}[section]
\newtheorem{definition}[theorem]{Definition}

\newtheorem{fact}[theorem]{Fact}
\newtheorem{cor}[theorem]{Corollary}
\newtheorem{lemma}[theorem]{Lemma}

\newtheorem{claim}[theorem]{Claim}

\DeclareMathOperator{\diam}{diam}

\newenvironment{proof}{\medskip \noindent{\sc Proof:}}{\quad$\Box$\par\medskip} 
\newenvironment{proof2}[1]{\medskip \noindent{\sc Proof of #1:}}{\quad$\Box$\par\medskip} 
\newcommand{\distribution}{D}
\def\gir{{\sf girth}}

\def\pis{{\pi^*}}
\def\pisk{{\pi^*_k}}

\definecolor{brwn}{RGB}{140, 70, 20}
\definecolor{gren}{RGB}{  0, 140, 10}

\begin{document}

\title{Optimal pebbling number of graphs with given minimum degree} 

\author{A. Czygrinow\thanks{School of Mathematical and Statistical Sciences, Arizona State University, aczygri@asu.edu.}
, G. Hurlbert\thanks{Department of Mathematics and Applied Mathematics, Virginia Commonwealth University, ghurlbert@vcu.edu.}, 
G. Y. Katona\thanks{Department of Computer Science and Information Theory, Budapest University of Technology and Economics, MTA-ELTE Numerical Analysis and Large Networks Research Group, Hungary, kiskat@cs.bme.hu.},
L. F. Papp\thanks{Department of Computer Science and Information Theory, Budapest University of Technology and Economics, MTA-ELTE Numerical Analysis and Large Networks Research Group, Hungary, lazsa@cs.bme.hu.}}
\maketitle 
 
\begin{abstract} Consider a distribution of pebbles on a connected graph $G$. A pebbling move removes two pebbles from a vertex and places one to an adjacent vertex. A vertex is reachable under a pebbling distribution if it has a pebble after the application of a sequence of pebbling moves. The optimal pebbling number $\pis(G)$ is the smallest number of pebbles which we can distribute in such a way that each vertex is reachable. 
It was known that the optimal pebbling number of any connected graph is at most $\frac{4n}{\delta+1}$, where $\delta$ is the minimum degree of the graph. We strengthen this bound by showing that equality cannot be attained and  that the bound is sharp. If $\diam(G)\geq 3$ then we  further improve the bound to $\pis(G)\leq\frac{3.75n}{\delta+1}$. On the other hand,  we show that for arbitrary large diameter and any $\epsilon>0$ there are infinitely many graphs whose optimal pebbling number is bigger than $\left(\frac{8}{3}-\epsilon\right)\frac{n}{(\delta+1)}$. 
\end{abstract}  

 \section{Introduction}
 Graph pebbling is a game on graphs initialized by a question of Saks and Lagarias, which was answered by Chung in 1989\cite{chung}. Its roots are originated in number theory. Each graph in this paper is simple. We denote the vertex set and the edge set of graph $G$ with $V(G)$ and $E(G)$, respectively. We use $n$ and $\delta$ for the order and the minimum degree of $G$, respectively. 

 A pebbling distribution $D$ on graph $G$ is a function mapping the vertex set to nonnegative integers. We can imagine that each vertex $v$ has $D(v)$  pebbles. A pebbling move removes two pebbles from a vertex and places one to an adjacent one. We do not want to violate the definition of pebbling distribution, therefore a pebbling move is allowed if and only if the vertex loosing pebbles has at least two pebbles. 
 
 A vertex $v$ is \emph{$k$-reachable} under a distribution $D$, if there is a sequence of pebbling moves, such that each move is allowed under the distribution obtained by the application of the previous moves and after the last move $v$ has at least $k$ pebbles. We say that a subgraph $H$ is \emph{k-solvable} under distribution $\distribution$ if each vertex of $H$ is k-reachable under $\distribution$. When the whole graph is k-solvable under a pebbling distribution, then we say that the distribution is k-solvable. 

When $S$ is a subset of the vertex set then let $D(S)$ denotes the total number of pebbles placed on the elements of $S$.
We say that $D(V(S))$ is the size of $D$. We use the standard $|X|$ notation to denote the size of $X$ when $X$ is a pebbling distribution or a set.     
A pebbling distribution $\distribution$ on a graph $G$ will be called \emph {$k$-optimal} if it is $k$-solvable and 
its size is the smallest possible. This size is called as the $k$-optimal pebbling number and denoted by $\pisk(G)$. When $k$ equals one we omit the $k$ part from all of the previous definitions and notations.

The optimal pebbling number of several graph families are known. For example exact values were given for paths and cycles \cite{path1,path2}, ladders \cite{ladder}, caterpillars \cite{caterpillar} and m-ary trees \cite{m-ary}. 
The values for graphs with diameter smaller than four are also characterized by some easily checkable domination conditions \cite{smalldiam}. However, determining the optimal pebbling number for a given graph is NP-hard \cite{NPhard}.    

In \cite{ladder} the optimal pebbling number of graphs with given minimal degree is studied. This paper contains many great results about the topic. The authors proved that $\pis(G)\leq\frac{4n}{(\delta+1)}$ and they also found a version utilizing the girth of the graph. A construction for infinite number of graphs with optimal pebbling number $(2.4-\frac{24}{5\delta+15}-o(\frac{1}{n}))\frac{n}{\delta+1}$ is also given in that article. 

In the present paper we continue the study of graphs with fixed minimum degree. We prove that there are infinitely many
diameter two graphs whose optimal pebbling number are close to the $\frac{4n}{\delta+1}$ upper bound. More precisely, for any $\epsilon>0$ there is a graph whose optimal pebbling number is more than $\frac{(4-\epsilon)}{\delta+1}$.

One can ask that what happens if we consider bigger diameter? In the second part of Section \ref{graphconstruction} we use the previous graphs as building blocks to construct a family of graphs with arbitrary large diameter, fixed minimum degree, and high optimal pebbling number. For any $d$ and $\epsilon>0$ we present a graph whose diameter is greater than $d$ and its optimal pebbling number is more than $\left(\frac{8}{3}-\epsilon\right)\frac{n}{(\delta+1)}$. 

In the case when the diameter is at least three we also prove a stronger upper bound in Section \ref{upperbound}. It is shown that $\pis(G)\leq \frac{15n}{4(\delta+1)}$ holds in this case. Unfortunately, we do not know that if it sharp or not, but this is enough to conclude that the original upper bound of \cite{ladder} does not hold with equality.

Finally, the authors of \cite{ladder} prove, for $k\ge 6$, that the family of connected graphs $G$ with $n$ vertices and minimum degree $k$ has $\pis(G)/n\rightarrow 0$ as $\gir(G)\rightarrow\infty$.
They asked if the same could be true for $k\in\{3,4,5\}$.
We prove in Section \ref{highgirth} that this is true for $k\in\{4,5\}$, leaving the case $k=3$ unresolved.

\section{A family of graphs with large optimal pebbling number}
\label{graphconstruction}

We say that a vertex $v$ is dominated by a set of vertices $S$, if $v$ is contained in $S$ or there is a vertex in $S$ which is adjacent to $v$. A vertex set $S$ dominates $G$ if each vertex of $G$ is dominated by $S$. An edge $\{u,v\}$ dominates $G$ if the set of its endpoints dominates $G$.  $G\square H$ denotes the Cartesian product of graphs $G$ and $H$, so $V(G\square H)=V(G)\times V(H)$ and $\{(g,h),(g',h')\}\in E(G\square H)$ if either $g=g'$, $\{h,h'\}\in E(H)$ or $\{g,g'\}\in E(G)$, $h=h'$.

The \emph{distance between vertices} $u$ and $w$ is the number of edges contained in a shortest path between them. We denote this quantity with $d(u,v)$. The \emph{distance-$k$ open neighborhood of a vertex} $v$, denoted by $N^k(v)$, contains all vertices whose distance from $v$ is exactly $k$. On the other hand, \emph{the distance-$k$ closed neighborhood} of $v$ contains all vertices whose distances from $v$ is at most $k$. We denote this set by $N^k[v]$. When $k=1$ we omit the distance-$1$ part from the name and the upper index $1$ from the notation.

We are going to use small graphs as a building blocks in our construction. We call these graphs \emph{special}. 

\begin{definition}
Graph $H$ is \emph{special} if the following properties are fulfilled:
\begin{enumerate}
\item The diameter of $H$ is two.
\item $H$ does not have a dominating edge.
\item $H$ has two vertices, called a \emph{special pair}, $u$ and $v$ such that:
\begin{enumerate}
\item Distance between $u$ and $v$ is two.
\item If we remove $u$ or $v$ from $H$ then the obtained subgraph does 
not have a dominating edge. 
\item $\left(N(u)\cap N(v)\right)\cup \{u,v\}$ is a dominating set of $H$.
\end{enumerate}
\end{enumerate}
\end{definition}

\begin{claim}
Let $K_m$ be the complete graph with $V(K_m)=\{x_1, \dots, x_m\}$. Then
$\overline{K_m\square K_m}$, the complement of $K_m\square K_m$, is special for any $m\geq 4$.
\end{claim}

\begin{proof}
Two vertices are adjacent in $\overline{K_m\square K_m}$ if and only if both of their coordinates are different. Its diameter is two since any two vertices share
a common neighbor.

If $(x_i,x_j)$ and $(x_k,x_l)$ are two  
adjacent vertices, then $i\neq k$, $j\neq l$, and so neither
$(x_i,x_l)$ nor $(x_k,x_j)$ is dominated by $\{(x_i,x_j),(x_k,x_l)\}$. Hence two adjacent vertices do not dominate $V(\overline{K_m\square K_m})$. Furthermore if we remove a vertex, at least one of the two not dominated vertices remains, which is still not dominated.

Let $u:=(x_1,x_1)$ and $v:=(x_1,x_2)$ be the special pair.
Then $(x_m, x_m)$ is a common neighbor of $u$ and $v$ and it dominates everything except $(x_i,x_m)$ and $(x_m,x_i)$ where $i<m$. Vertices $u$ and $v$ dominate all of these except $(x_1,x_m)$, but this last one is dominated by $(x_2,x_3)$ which is also a common neighbor of $u$ and $v$.
Therefore $\left(N(u)\cap N(v)\right)~\cup~\{u,v\}$ is a dominating set of $\overline{K_m\square K_m}$. 
\end{proof}

There are many other special graphs. In this paper we mention just one more example: the circulant graph where the vertices are labeled from $1$ to $2m$, $m\geq 5$ and two vertices $i$ and $j$ are adjacent if
and only if $i-j \not\equiv m, m\pm 1 \mod 2m$.

Muntz \textit{et al.} showed that the optimal pebbling number of a diameter two graph lacking a dominating edge is $4$. Therefore the optimal pebbling number of any special graph is $4$. Using this we can prove the following theorem:

\begin{theorem}
\label{randomgraphproof}
For any $\epsilon>0$ there is a graph $G$ on $n$ vertices with diameter two such that $\pis(G)> \frac{(4-\epsilon)n}{\delta+1}$.
\end{theorem}

\begin{proof}
For a fixed $\epsilon>0$ consider a graph $\overline{K_m\square K_m}$ where $m>\max(\frac{a}{a-1},2)$ and $a=\sqrt{\frac{4}{4-\epsilon}}$. Each degree in $\overline{K_m\square K_m}$ is $(m-1)^2$ and its order is $m^2$. Since both $a$ and $m$ are greater than $1$ we have that $a>\frac{m}{m-1}$, which implies the following:
$$\pis(\overline{K_m\square K_m})=4=(4-\epsilon)a^2>\frac{(4-\epsilon)m^2}{(m-1)^2}>\frac{(4-\epsilon)n}{\delta+1}$$
\end{proof}

Now we turn our attention to the case when the diameter is at least three. For any $d$ and $\epsilon$ we show a graph whose diameter is $9d-1$ and its optimal pebbling number is bigger than $\left(\frac{8}{3}-\epsilon\right)\frac{n}{\delta+1}$. In fact we connect $3d$ special graphs sequentially using the special pairs and prove that the optimal pebbling number of the obtained graph is $8d$. To obtain the desired result we choose the special graphs to be $\overline{K_m\square K_m}$ again. 

Let $H_1$, $H_2,\dots,H_l$ be (not necessarily the same) special graphs. Denote the vertex set of $H_i$ by $B_i$ and let $u_i$ and $v_i$ be a special pair of $H_i$. Connect these graphs sequentially by placing edges between $v_i$ and $u_{i+1}$ to obtain a new graph $G_l$. Consequently  $V(G_l)=\cup_{i=1}^l~B_i$ and $E(G_l)=\cup_{i=1}^lE(H_i)\bigcup\cup_{i=1}^{l-1}\{v_i,u_{i+1}\}$.
We say that a $H_i$ subgraph of $G_l$ is a \emph{block}.

\begin{claim}
Let $k\in Z^+$,  $r \in \{0,1,2\}$, and let $G_{3k+r}$ be the graph defined above. Then 
$$\pis(G_{3k+r})\leq \begin{cases}
8k &\text{if }r=0\\
8k+4 &\text{if }r=1\\
8k+6 &\text{if }r=2. 
\end{cases}$$

\end{claim}

\begin{proof}
To construct a solvable distribution with $8k$ pebbles in the $r=0$ case we place $4$ pebbles at vertices $v_{3j-2}$ and $u_{3j}$ where $j=1,2,\dots,k$.
If a vertex of an $H_i$ has 4 pebbles, then each vertex of $H_i$ is reachable because $H_i$ is a diameter two graph. Otherwise $u_i$ and $v_i$
are both adjacent to vertices having 4 pebbles. Therefore each element of $\left(N(u_i)\cap N(v_i)\right)\cup \{u_i,v_i\}$ is 2-reachable and since
this is a domination set of $H_i$, we can move a pebble to any vertex of $H_i$.
When $r=1$ we use the same construction and place $4$ additional pebbles at a vertex of $H_{3k+1}$, which solves $H_{3k+1}$. In the last $r=2$ case we start again with the $r=0$ construction and place $3$ pebbles at both $v_{3k+1}$ and $u_{3k+2}$. These two vertices are 4-reachable, therefore all vertices in the last two blocks are solvable.
\end{proof}

\begin{claim}
Let $G_l$ be a graph defined above, then $\pis(G_1)=4$, $\pis(G_2)=6$, $\pis(G_3)=8$.
\label{small_lower}
\end{claim}

\begin{proof}
The upper bounds have been shown. $G_1$ is a special graph, therefore its optimal pebbling number is $4$. 

Let $D$ be a pebbling distribution of $G_2$ having $5$ pebbles. We can assume that $H_1$ has less pebbles than $H_2$, therefore $D(B_1)\leq 2$. Since $\pis(H_1)=4$, we can not reach each vertex of $H_1$ without the usage of the pebbles placed at $B_2$. $H_1$ and $H_2$ are connected by only one edge, therefore the only way to use the pebbles of $B_2$ is to move as many pebbles through this edge as possible. Furthermore, if we do nothing with the pebbles of $B_1$ and we move as many pebbles from $B_2$ to $v_1$ as possible, then the obtained distribution of $H_1$ has to be solvable if $D$ is solvable on $G_2$. But this obtained distribution can have at most $D(B_1)+ D(B_2)/2\leq 3$ pebbles, hence $D$ is not solvable.

To prove the last statement we argue that a distribution with 
7 pebbles is not solvable on $G_3$.  Let $l:=\distribution(B_1)$. If $l\geq 3$, then we can accumulate on $B_2 \cup B_3$ 
at most $l/2+ (7-l)<6$ pebbles. 
If $l=0$, then 4 pebbles must be obtained from 7 on $B_2\cup B_3$ which is not 
possible and if $l=1$, then to be able to solve for $B_1$ the remaining 6 pebbles 
must placed on $u_2$ which is not solvable for $B_3$. Finally if $\distribution(B_1)=2=\distribution(B_3)$, 
then only 3 pebbles are on $B_2$ and it is not  possible to move additional two to $B_1$.
\end{proof}

\begin{claim}
\label{diamlower}
Let $G_l$ be a graph defined above, then $\pis(G_l)\geq\frac{8}{3}l$.
\end{claim}

To prove this claim we need some preparation. We are going to use a cut argument which requires that in an optimal distribution several edges can not transfer pebbles. Therefore these edges can be removed from the graph without changing the optimal pebbling number. We are going to remove edges in such a way that the obtained graphs has two connected components and each of them is a smaller instance of $G_l$ or it is
almost a $G_l$.

Let $G_l^-$ be the subgraph of $G_{l}$ which we obtain by deleting $v_l$. Let $G_l^+$ be the following graph: We take $G_{l}$ and add a leaf to $u_1$. Let $G_0^+$ be the one vertex graph.

\begin{lemma}
\label{existsnot2reachable}
Let $D$ be a solvable distribution $D$ on $G_l$ such that $|D|<3l-1$. Then there is a decomposition of $G_l$ into parts isomorphic to either $G_k$ and $G_{l-k}$, with $1\leq k<l$, or $G_k^-$ and $G_{l-k}^+$, with $1\leq k \leq l$, such that no pebbles can be moved through the edges connecting the two parts.
\end{lemma}

We are going to use the \emph{collapsing technique} which was introduced in \cite{ladder}.
Let $G$ and $H$ be simple graphs. We say that $H$ is a \emph{quotient} of $G$, if there is a surjective mapping $\phi: V(G)\rightarrow V(H)$ such that $\{h_1,h_2\}\in E(H)$ if and only if there are $g_1,g_2\in V(G)$, where $\{g_1,g_2\}\in E(G)$, $h_1=\phi(g_1)$ and $h_2=\phi(g_2)$. We say that $\phi$ collapses $G$ to $H$, and if $D$ is a pebbling distribution on $G$, then the \emph{collapsed distribution} $D_{\phi}$ on $H$ is defined in the following way:
$D_{\phi}(h)=\sum_{g\in V(G)| \phi(g)=h}D(g)$. 

\begin{claim}
\label{collapsing}
If $v$ is $k$-reachable under $D$, then $\phi(v)$ is $k$-reachable under $D_{\phi}$.
\end{claim}

This claim is a generalization of the Collapsing Lemma \cite{ladder}. The proof given by Bunde \textit{et al.} can be used to prove our claim. 
We are going to collapse $G_l$ to the path containing $3l$ vertices, denoted by $P_{3l}$, therefore we state some results about the path. 

\begin{theorem}[\cite{ladder}]
A 2-optimal distribution of the $n$-vertex path contains $n+1$ pebbles.
\label{2optimal}
\end{theorem}

\begin{claim}
\label{2not2reachable}
If an inner vertex of the path is not 2-reachable, then one of its neighbors is not either.   
\end{claim}

\begin{proof}
If a vertex $v$ is not $2$-reachable, then $v$ cuts the path to two, and no pebble can be moved through $v$. If both neighbors of $v$ are 2-reachable, then we can move 2 pebbles to them simultaneously, since $v$ can receive two pebbles in total, one from each neighbor, which is a contradiction. 
\end{proof}

\begin{proof2}{Lemma \ref{existsnot2reachable}} 
Let $D'$ be the following distribution on $G_l$: $D'(u_1)=D(u_1)+1$, $D'(v_l)=D(v_l)+1$ and $D'(x)=D(x)$ if $x\notin\{u_1,v_l\}$.
Since $D$ is a solvable distribution, $u_1$ and $v_l$ are 2-reachable under $D'$. 

Let $P_{3l}$ be a path on $3l$ vertices and denote its vertices with $p_1,p_2,\dots p_{3l}$.
Let $\phi$ be a mapping which maps $G_l$ to $P_{3l}$ in such a way that $\phi(u_i)=p_{3i-2}$, $\phi(v_i)=p_{3i}$ and $\phi(x)=p_{3i-1}$ if $x\in B_i\setminus\{u_i,v_i\} $. $P_{3l}$ is a quotient of $G_l$. 

$D'_{\phi}$ has less than $3l+1$ pebbles, therefore according to Theorem~\ref{2optimal} it is not a 2-solvable distribution of $P_{3l}$. 
There is a vertex $p_i$ in $P_{3l}$ which is not $2$-reachable under $D'_{\phi}$. 
Claim \ref{collapsing} implies that both $p_1$ and $p_{3l}$ are $2$-reachable, so $1<i<3l$.
Because $p_i$ is not $2$-reachable, no pebble can move from $p_i$ to either $p_{i-1}$ or $p_{i+1}$.
Similarly, pebbles cannot move from both $p_{i-1}$ and $p_{i+1}$ to $p_i$.
Thus, for one of the edges incident with $p_i$, without loss of generality $\{p_i,p_{i+1}\}$, no pebble can move across it.
This also means that $p_{i+1}$ is not $2$-reachable.

Claim ~\ref{collapsing} 
yields that the vertices of $\phi^{-1}(p_i)\cup \phi^{-1}(p_{i+1})$ are also not 2-reachable under $D$.
Therefore no pebbles can be passed between $\phi^{-1}(p_i)$ and $\phi^{-1}(p_{i+1})$. Deleting the edges between them makes
the graph disconnected and leaves $D$ solvable.   

The two connected components are isomorphic to either $G_k$ and $G_{l-k}$ or $G_k^-$ and $G_{l-k}^+$, where $i\in\{3k,3k-1,3(l-k)+1\}$.
This comes from the collapsing function.
\end{proof2}

Lemma \ref{existsnot2reachable} guides us to make an induction argument. However we also need some information about $\pis(G_l^+)$ and
$\pis(G_l^-)$. The 3(b) property of special graphs guarantees that these values are at least 
$\pis(G_l)$.

\begin{lemma}
$\pis(G_l^+)\geq \pis(G_l)$ and $\pis(G_l^-)\geq \pis(G_l)$. 
\label{bigger}
\end{lemma}

\begin{proof}
Adding a leaf can not decrease the optimal pebbling number, hence the first inequality holds. 
$G_1^-$ is a diameter two graph without a dominating edge, therefore $\pis(G_1^-)=4=\pis(G_1)$.
To prove the rest of the assertion, we show that there is an optimal distribution $D'$ of $G_l^-$ which is also a solvable distribution on $G_l$ (where we interpret $D'$ to be a distribution on $G_l$ with no pebbles on $v_l$). 

Let $D$ be an optimal distribution of $G_l^-$.
Denote the last block of $G_l^-$, where the removed vertex was located, by $H_l'$. Since $H_l$ was special, $H_l'$ does not have a dominating edge and its diameter is at least two, therefore $\pis(H'_l)\geq 4$.   

$H_l'$ can obtain pebbles from the rest of the graph only through the $v_{l-1},u_l$ edge. Let $k$ be
the maximum number of pebbles which can arrive at $u_l$ using this edge. Since $D$ was solvable, $k+D(V(H_l'))\geq 4$. If we relocate all pebbles of $H_l'$ at $u_l$ and put back $v_l$, the obtained $D'$ distribution of $G_l$ also satisfies that $k+D'(u_l)\geq 4$, therefore each vertex of $B_l$ is reachable under $D'$.
The other vertices remained reachable since we moved the pebbles closer to them in one pile. 
\end{proof}

\begin{proof2}{Claim \ref{diamlower}}
Assume the contrary. Let $G_l$ be a minimal counterexample. Claim \ref{small_lower} implies $l\geq 4$.
Let $D$ be an optimal distribution of $G_l$. 

Since $G_l$ is a counterexample, $|D|<\frac{8}{3}l\leq 3l-1$. Therefore we can apply Lemma \ref{existsnot2reachable}.
According to this lemma, we can break $G_l$ to $G_k$ and $G_{l-k}$ or $G_k^-$ and $G_{l-k}^+$ such that no pebbles can be moved
between the two parts. This means that $D$ induces solvable distributions on both parts. 

In the second case when $k\neq l$, using Lemma \ref{bigger} gives the following chain of inequalities:

$$\pis(G_l)= D(G_k^-)+D(G_{l-k}^+)\geq \pis(G_k^-)+\pis(G_{l-k}^+)\geq\pis(G_k)+ \pis(G_{l-k})$$

$G_l$ was a minimal counterexample, therefore:

$$\pis(G_l)
\geq\pis(G_k)+ \pis(G_{l-k})\geq \frac{8}{3}k+\frac{8}{3} (l-k)=\frac{8}{3}l$$

This contradicts with our assumption. When $l=k$ we have:
$$\pis(G_l)= D(G_l^-)+D(G_{0}^+)\geq \pis(G_l^-)+\pis(G_{0}^+)\geq\pis(G_l)+ 1,$$
which is also a contradiction.
Therefore there is no counterexample.
\end{proof2}

\begin{cor}
$\pis(G_{3k})=8k$
\end{cor}

\begin{theorem}
For any $\epsilon>0$ and any integer $d$, there is a graph $G$, such that its diameter is bigger than $d$ and $\pis(G)\geq (\frac{8}{3}-\epsilon)\frac{n}{\delta+1}$.
\end{theorem}

\begin{proof}
Consider $G_{3d}$ with $\overline{K_m\square K_m}$ blocks, where $m>\max(\frac{a}{a-1},3)$ and $a=\sqrt{\frac{8/3}{8/3-\epsilon}}$. Its diameter is $3d-1$, $\delta(G_{3d})=\delta(\overline{K_m\square K_m})=(m-1)^2$ and $|V(G_{3d})|=3dm^2$. The optimal pebbling number of $G_{3d}$ is $8d$. If we repeat the calculation of Theorem \ref{randomgraphproof} we receive the desired result:

$$\pis(G_{3d})=8d=3d\left(\frac{8}{3}-\epsilon\right)
a^2>3d\left(\frac{8}{3}-\epsilon\right)\frac{m^2}{(m-1)^2}>\left(\frac{8}{3}-\epsilon\right)\frac{n}{\delta+1}$$
\end{proof}

\section{Improved upper bound when diameter is at least three}
\label{upperbound}

In this section we give a construction of a pebbling distribution having at most $\frac{15n}{4(\delta+1)}$ pebbles for any graph whose diameter is at least three.

We are going to talk about several graphs on the same labeled vertex set. To make it clear which graph we are considering in a formula we write the name of the graph as a lower index, i.e. $d_G(u,v)$ is the distance between vertices $u$ and $v$ in graph $G$. 

We define distances between subgraphs in the natural way: If $H$ and $K$ are subgraphs of $G$, then $d_G(H,K)=\min_{u\in V(H),v\in V(K)}(d_G(u,v))$. 

We can think about a vertex as a subgraph, therefore let \emph{distance-$k$ open neighborhood of a subgraph} $H$ be the set of vertices whose distance from $H$ is exactly $k$. We define the closed neighborhood similarly. Note that $N^d(H)=N^d[H]\backslash N^{d-1}[H]$.

The following property will be useful in our investigations:
A vertex $v\in V(G)$ is \emph{strongly reachable} under $D$ if each vertex from the closed neighborhood of $v$ is reachable under $D$. This property together with traditional reachability partition the vertex set to three sets $\mathcal{T}(D)$, $\mathcal{H}(D)$ and $\mathcal{U}(D)$, where $\mathcal{T}(D)$ includes the strongly reachable vertices, the vertices of $\mathcal{H}(D)$ are reachable but not strongly reachable and $\mathcal{U}(D)$ contains the rest of the vertices. 

\begin{theorem}
Let $G$ be a connected graph, such that its diameter is bigger than two and $\delta$ is its minimum degree. We have
$$\pis(G)\leq\frac{15 n}{4(\delta+1)}.$$
\label{improvedbound}
\end{theorem}

Let $D$ and $D'$ be pebble distributions. $D'$ is an \textit{expansion} of $D'$ ($D\leq D'$) if $\forall v \in V(G)$ $D(v)\leq D'(v )$. If $D\neq D'$, then we write that $D<D'$. If $D'$ is an expansion of $D$, then let $\Delta_{D,D'}$ be the pebbling distribution defined as: $\Delta_{D,D'}(v)=D'(v)-D(v)$ $\forall v\in V(G)$.

If we would like to create a solvable distribution, then we can do it incrementally. We start with the trivial distribution with no pebbles and add more and more pebbles to it. So we have a sequence of distributions $0< D_1< D_2<\dots <D_{k-1}<D_k$ where $D_k$ is solvable. The number of reachable vertices is growing during this process. We can ask  which vertices are reachable, strongly reachable, or not reachable after the $i$th step. Let $\mathcal{T}(D_i)$, $\mathcal{H}(D_i)$, $\mathcal{U}(D_i)$ denote these sets respectively. 
Note that $\mathcal{T}(D_i)\subseteq \mathcal{T}(D_{i+1})$, while $\mathcal{U}(D_i)\supseteq\mathcal{U}(D_{i+1})$.
Furthermore we know that $\mathcal{T}(D_k)=V(S)$ and $\mathcal{H}(D_k)=\mathcal{U}(D_k)=\emptyset$. 

 \par If for each $i$ the difference $|\Delta_{D_i,D_{i+1}}|$ is relatively small and $|\mathcal{T}(D_{i+1})\setminus\mathcal{T}(D_i)|$ is relatively big, then it yields that $|D_k|$ is not so big. 
 
To make this intuitive idea precise we define the $\emph{strengthening ratio}$. 

\begin{definition}
Suppose that we have distributions $D$ and $D'$ on graph $G$, such that $D< D'$. Denote the difference of the size of these distribution by  $\Delta p_{D,D'}=|D'|-|D|=|\Delta_{D,D'}|$. 
We use  $\Delta \mathcal{T}_{D,D'}$ for set $\mathcal{T}(D') \setminus \mathcal{T}(D)$ and $\Delta t_{D,D'}$ denotes the cardinality of this set. 

We say that the \emph{strengthening ratio of the expansion} $D<D'$ is:
$$\mathcal{E}(D,D')=\frac{\Delta t_{D,D'}}{\Delta p_{D,D'}}$$
The \emph{strengthening ratio of distribution} $D\neq 0$ is $\mathcal{E}(0,D)$, and the strengthening ratio of $D=0$ is $\infty$.
\end{definition}

\begin{fact}
If $D$ is solvable, then $|D|=\frac{n}{\mathcal{E}(0,D)}$. 
\label{solvfact}
\end{fact}

\par
This fact shows that if we want to give a solvable distribution whose
size is close to the optimum, then its strengthening ratio is also
close to the optimum. Furthermore, a smaller solvable distribution has
bigger strengthening ratio. The next lemma shows that if we break
$D_k$ to a sequence of expansions $0<D_1<D_2<\dots <D_{k-1}<D_k$, then the
strengthening ratio of each expansion is a lower bound for
$\mathcal{E}(0,D_k)$.
Therefore we are looking for an expansion chain where the minimum strengthening ratio among all expansion steps is relatively big. 

\begin{lemma}
Let $D_1$, $D_2$ and $D_3$ are distributions on $G$. If $D_1< D_2$ and $D_2< D_3$, then
$$\mathcal{E}(D_1,D_3)\geq \min(\mathcal{E}(D_1,D_2),\mathcal{E}(D_2,D_3) ). $$ 
\end{lemma}

\begin{proof}

Let $a,b,c,d$ be nonnegative real numbers, then the following inequality can be easily proven by elementary tools:
$$ \frac{a+b}{c+d}\geq\min\left (\frac{a}{c},\frac{b}{d}\right ).$$ 

Using this and the definition of strengthening ratio, we obtain
\begin{equation*}
\begin{split}
\mathcal{E}(D_1,D_3)&=\frac{\Delta t_{D_1,D_3}}{\Delta p_{D_1,D_3}}=\frac{\Delta t_{D_1,D_2}+\Delta t_{D_2,D_3}}{\Delta p_{D_1,D_2}+\Delta p_{D_2,D_3}}\geq {}\\
&\geq  \min\left (   \frac{\Delta t_{D_1,D_2}}{\Delta p_{D_1,D_2}},\frac{\Delta t_{D_2,D_3}}{\Delta p_{D_2,D_3}}\right )
 = \min \left ( \mathcal{E}(D_1,D_2),\mathcal{E}(D_2,D_3)   \right ).
\end{split}
\end{equation*}

\end{proof}

\par
In the next lemma we state that we can construct a distribution $D$ with
some special properties. This lemma formalizes the following idea: If
there are pairs of adjacent vertices, such that the closed
neighborhood of each pair is large, then we can make all vertices of
these pairs reachable with few pebbles, while lots of other vertices
become reachable. The connection between few and lots of is
established by strengthening ratio.
 
\begin{lemma}
\label{startdistribution}
Let $G$ be an arbitrary simple connected graph. There is a pebbling
distribution $D$ on $G$ which satisfies the following conditions:
\begin{enumerate}
\item The strengthening ratio of $D$ is at least
  $\frac{4}{15}(\delta+1)$.
\item If $(u,v)$ is an edge of $G$ and $|N[u]\cup
  N[v]|\geq\frac{29}{15}(\delta+1)$, then both of $u$ and $v$ are
  reachable under $D$.
\end{enumerate}
\end{lemma}

\begin{proof}
Our proof is a construction for such a $D$:
\par
We say that and edge $(u,v)$ has * property if and only if $|N[u]\cup N[v]|\geq\frac{29}{15}(\delta+1)$.
\par
First of all, if there is no $(u,v)$ edge in $G$ with * property, then
the trivial distribution $0$ is good to be $D$. Otherwise, we have to
make reachable each vertex of any edge which has * property. To make
this we search for these edges, and if we find such an edge such that at
least one of its vertices is not reachable, then we add some pebbles
on $D$ to make it reachable.
\par We will define sets $H,A,B\subset V(G)$, $P,R\subset V(G)\square V(G)$ and let
$L_p$ be a set containing vertices of $G$ for each $p\in P$.

\par These sets, except $H$, will contain the edges with * property or
their vertices. They will have the following semantics at the end of
the construction:
\begin{itemize}
\item Each element of $H$ will be reachable under $D$, but not necessarily all of the reachable vertices contained in it.
\item Each vertex of $B$ has a neighbor who has at least two 4-reachable distance-2 neighbors, or has an 8-reachable distance-3 neighbor.
\item The elements of $P$ are edges whose vertices will be 4-reachable. 
\item The elements of $R$ are edges whose vertices will be 8-reachable.
\item $L_p$ contains vertices from $A$ whose distance from $p$ is exactly 3.
\end{itemize}
Then do the following:
\begin{enumerate}
\item Choose an edge $(u,v)$ which has * property and $u,v\notin H$. If we can not choose such an edge, then move to step 3.
\item Add the elements of $N^2[u]\cup N^2[v]$ to $H$. Add $(u,v)$ to $P$. Move to step 1. 
\item Search for an edge $(u,v)$  which has * property and $v\notin H$. If we can not find one, then move to step 6.
\item Add the elements of $N^2[v]$ to $H$.
\item Count the number of pairs $p$ in $P$ whose distance from $u$ is 2. If we get more than one, then add $v$ to $B$ and move to step 3. Otherwise, add $v$ to $A$ and add $v$ to the set $L_p$ where $p$ is the only pair whose distance from $u$ is 2.
\item Do for each $p\in P$: If $|L_p|\geq 5$, then move the elements of $L_p$ from $A$ to $B$ and also move $p$ from $P$ to $R$. 
\item Let $D$ be the following:
$$D(v)=\begin{cases}
4 & \text{if } v\in A,\\
3 & \text{if either } v\in B\text{ or } v \text{ is an element of pair }p\in P,\\
5 & \text{if } v \text{ is the first element of pair } p\in R, \\
6 & \text{if } v \text{ is the second element of pair } p\in R,\\ 
0 & \text{otherwise}.
\end{cases}$$   

\end{enumerate}
\par First, if we choose an edge with * property, then both vertices of it are reachable under $D$. To see this consider $H$.  Each vertex of $H$ is reachable under $D$ by construction. We expanded $H$ by distance-2 closed neighborhoods of vertices which are 4-reachable in each step. Each vertex of an edge with * property is contained in $H$, $A$ or $B$. 
\par
 Hence the second condition is satisfied. So we just need to verify the first one.
\par The vertices of sets $A$, $B$, and vertices of edges contained in $P$ and $R$ are all 4-reachable. Hence each vertex belongs to their neighborhood is strongly reachable. This implies that:
\begin{equation*}
\Delta t_{0,D}=|\mathcal{T}(D)| \geq  \left |\left(\bigcup_{p\in P\cup R} N[p]\right)\bigcup \left(\bigcup_{v\in A\cup B} N[v]\right)\right |= \sum_{p\in P\cup R} |N[p]|+\sum_{v\in A\cup B} |N[v]|
\end{equation*}
For the second equality we need that these neighborhoods are disjoint, but this is true because of the construction: The distance between a vertex of $A\cup B$ and a pair $p$ of $P\cup R$ is at least 3. The distance between $p,p'\in P\cup R$ is also at least 3. Both of these are guaranteed by step 2. $d(u,v)\geq 3$ where $u,v\in A \cup B$ because of step 4. 
\par Using the * property of edges contained in $P$ and $R$ gives:

\begin{equation*}
\Delta t_{0,D}\geq \sum_{p\in P\cup R} |N[p]|+\sum_{v\in A\cup B} |N[v]|\geq (|P|+|R|)\cdot \frac{29}{15}(\delta+1) + (|A|+|B|)(\delta+1),
\end{equation*}

\begin{equation*}
\Delta p_{0,D}=|D|=4|A|+3|B|+6|P|+11|R|,
\end{equation*}
\begin{equation*}
\mathcal{E}(0,D)=\frac{\Delta t_{0,D}}{\Delta p_{0,D}}\geq \frac{(|P|+|R|)\cdot \frac{29}{15}(\delta+1) + (|A|+|B|)(\delta+1)}{4|A|+3|B|+6|P|+11|R|}.
\end{equation*}
Using  $(a+b)/(c+d)\geq \min(a/c,b/d)$:
\begin{equation*}
\mathcal{E}(0,D) \geq \min\left ( \frac{(\frac{29}{15}|P|+|A|)(\delta+1)}{6|P|+4|A|}, \frac{(\frac{29}{15}|R| +|B|)(\delta+1)}{11|R|+3|B|}  \right ).
\end{equation*}
Step 6 of the construction implies that $|A|\leq 4|P|$ and $|B|\geq 5|R|$.
\par Let $|A|=4x|P|$. In this case $0\leq x\leq 1$ and we get the following function of $x$ for the first part, which gains its minimum at $x=1$:

$$\frac{(\frac{29}{15}|P|+|A|)(\delta+1)}{6|P|+4|A|}=\frac{(\frac{29}{15}+4x)(\delta+1)}{6+16x}\geq \frac{89}{330}(\delta+1).$$ 
\par Let $|B|=5y|R|$. $|B|\geq 5|R|$ implies that $1\leq y$. The function which we get from the second part gains its minimum at $y=1$:
$$\frac{(\frac{29}{15}|R| +|B|)(\delta+1)}{11|R|+3|B|}=\frac{(\frac{29}{15}+5y)(\delta+1)}{11+15y}\geq \frac{4}{15}(\delta+1).$$ 

This completes the proof of the Lemma.
\end{proof}

During the proof we will show that a non solvable distribution whose strengthening ratio is above 
the desired bound always can be expanded to a bigger one whose strengthening ratio is still reasonable. To do this we want to decrease the number of vertices which are not strongly reachable. Usually we place some pebbles at not reachable vertices. We know that if a vertex $v$ is not reachable under $D$ and we make it 4-reachable, then all vertices of its closed neighborhood, which were not strongly reachable, become strongly reachable.

\par We usually consider a connected component $S$ of the graph induced by $\mathcal{U}(D)$. 

\par There are several reasons why we do this. First of all, a chosen $S$ is a small connected part of $G$ where none of the vertices are reachable, hence it is much simpler to work with $S$ instead of the whole graph.

\par
A vertex from $S$ has the property that none of its neighbors are strongly reachable. Thus, if we make a vertex from $S$ 4-reachable, then its whole closed neighborhood becomes strongly reachable.

\par Another reason for considering such an $S$ is that if we add some additional pebbles to $S$ and make sure that all of its vertices become reachable, then these vertices become strongly reachable, too.

\par If we make $u$ and $v$ both 4-reachable with at most 7 pebbles and their closed neighborhoods are disjoint then this is good for us. The disjointness of the neighborhoods happens when $d(v,u)\geq 3$.

\par We said that we want to investigate $S$, which is a connected component of $\mathcal{U}(D)$. On the one hand, it is beneficial, but on the other hand it makes some trouble when we consider distances. Let $u$ and $v$ be vertices of $S$. Their distance can be different in $G$ and $S$. For example if $G$ is the wheel graph on $n$ vertices and we place just one pebble at the center vertex, then $S$ is the outer circle and the distance between two vertices of $S$ can be $\left\lfloor\frac{n-1}{2}\right\rfloor$, while their distance in $G$ is not larger than 2.

\par This difference is important because this shows that we can not decide the disjointness of closed neighborhoods by distance induced by $S$. The first idea to handle this is considering the original distance given by $G$, but then we have to consider the whole graph, which we would like to avoid. To overcome this problem we make the following compromise:

\par We count distances in graph $N[S]$. Clearly, this distance also can be smaller than the corresponding distance in $G$, but it happens only for values higher than 3. Hence this $N[S]$ distance determine disjointness of the neighborhoods, and it will be enough for our investigation.  

\par The following lemmas will be used in the proof.

\begin{fact}
\label{teny}
Let $S$ and $B$ be induced subgraphs of $G$ such that $V(B)=N_G[V(S)]$. If $\max_{u,v\in V(S)}(d_B(u,v))$ $= 3$, then either there exist vertices $a,b,c,d\in V(S)$,  such that they are neighbors in this order and $d_B(a,d)=3$, 
or there exist vertices $a,d\in V(S)$ such that $d_B(a,d)=3$ and there is a path $P$ between $a$ and $d$ whose length is 3 and $P$  contains a vertex from $V(B)\setminus V(S)$.  
\end{fact}

\begin{lemma}
\label{osszefugg}
Let $\delta$ be the minimum degree of graph $G$. Let $S$ and $B$ are connected induced subgraphs of $G$, such that $V(B)=N_G[V(S)]$. If $\max_{u,v\in V(S)}(d_B(u,v))=3$ and exist $a,d\in V(S)$, such that $d_B(a,d)=d_S(a,d)=3$, then there is an $u,v$ edge in $S$ whose closed neighborhood has size at least $\frac{4}{3}(\delta+1)$ . 
\end{lemma}

\begin{proof}
Let $a,b,c,d$ be the vertices of a shortest path between $a$ and $d$ which lies in $S$. If the statement holds for edge $(a,b)$ or $(c,d)$, then we have found the edge which we are looking for. Thus assume the contrary. The Inclusion-exclusion principle gives us the following result for vertex pair $a,b$:
$$|N[a]\cap N[b]|=\underbrace{|N[a]|}_{\geq \delta+1}+\underbrace{|N[b]|}_{\geq \delta+1}-\underbrace{|N[a]\cup N[b]|}_{<\frac{4}{3}\delta +\frac{4}{3} }>  \frac{2}{3}\delta+\frac{2}{3}.$$
The same is true for pair $c,d$.
\par The distance of $a,d$ implies that $N[a]\cap N[d]=\emptyset$. Thus $(N[a]\cap N[b])\cap (N[c]\cap N[d])=\emptyset$ 
\begin{equation*}
\begin{split}
|N[b]\cup N[c]|\geq& |(N[b]\cap N[a])\cup(N[c] \cap N[d])|=\\
=&|N[b]\cap N[a]|+|N[c]\cap N[d]|-|(N[a]\cap N[b])\cap (N[c]\cap N[d])|> \\
>&2\left (\frac{2}{3}\delta+\frac{2}{3}\right )-0=\frac{4}{3}\delta+\frac{4}{3}.
\end{split}
\end{equation*}
So edge $b,c$ has the required property.
\end{proof}

\begin{lemma}
Let $S$ and $B$ be connected induced subgraphs of $G$, such that $V(B)=N_G[V(S)]$. Assume that there are vertices $u$ and $v$ in $S$ whose distance in $B$ is 4. Then at least one of the following conditions holds:
\begin{enumerate}
\item There exists $a,b\in V(S)$, such that $d_S(a,b)=d_B(a,b)=4$
\item There exists $c,d\in S$, such that $d_B(c,d)=3$ and some of the shortest paths between $c$,$d$ contain a vertex from set $V(B)\setminus V(S)$.
\end{enumerate}
\label{diam4van}
\end{lemma}

\begin{proof}
Consider a pair of vertices $u,v\in V(S)$ whose distance in $B$ is four. It is clear that $d_S(u,v)\geq 4$. Equality means that the first condition is fulfilled. Assume that the distance in $S$ between $u$ and $v$ is greater than four. Let $P$ be a shortest path between $u$ and $v$ which lies in $S$. The length of $P$ is at least five. Label the vertices of $P$ as $u=p_0, p_1, p_2, \dots, p_k=v$. Let $i$ be the smallest value such $p_i$ does not have a neighbor in $N_B[u]$. 
The minimality of $i$ implies that $d_B(u,p_i)=3$. If $i>3$, then the shortest path between $u$ and $p_i$, which has length three, has to contain a vertex from $V(B)\setminus V(S)$. This gives us the second condition.
\par Otherwise $i=3$. Let $j$ be the smallest value, such that $p_j$ does not have a neighbor in $N_B^2[u]$. The case $j=4$ gives us $d_B(u,p_j)=4=d_S(u,p_j)$ which fulfills the first condition. The other case is $j>4$, when $d_B(p_0,p_4)=3$. It can happen if and only if the second condition holds.  
\end{proof}

\begin{lemma}
\label{diam4}
Let $\delta$ be the minimum degree of $G$, $S$ and $B$ be induced subgraphs of $G$, such that $V(B)=N[V(S)]$. If $\max_{u,v\in V(S)}(d_B(u,v))\geq 4$ and exist vertices $a,e\in V(S) $ such that $d_B(a,e)=d_S(a,e)=4$, then one of the following two conditions holds:
\begin{enumerate}
\item There exist $u,v\in S$ such that $d_B(u,v)=2$ and $|N_B[u]\cup N_B[v]|\geq \frac{28}{15}(\delta +1) $.\label{itm:first}
\item $|N[a]\cup N[e]\cup (N[b]\cap N[d])|\geq \frac{32}{15}(\delta+1)$, where $a,b,c,d,e$ are the vertices of a path lying in $S$.  \label{itm:second}
\end{enumerate}
 \end{lemma}

\begin{proof}
Assume that \ref{itm:first}. does not hold. This gives us the following estimate on the size of the common neighborhood of $b$ and $d$:
$$|N_G[b]\cap N_G[d]|=\underbrace{|N_G[b]|}_{\geq \delta+1}+\underbrace{|N_G[d]|}_{\geq \delta+1}-\underbrace{|N_G[b]\cup N_G[d]|}_{<\frac{28}{15}(\delta +1)}> \frac{2}{15}(\delta +1) $$
since $a$ and $d$ do not have a common neighbor. The same is true for the pairs of $b$, $e$, and $a$, $e$  which implies:
$$|N[a]\cup N[e]\cup (N[b]\cap N[d])|=|N[a]|+|N[e]|+|N[b]\cap N[d]|\geq \frac{32}{15}(\delta+1).$$
So if \ref{itm:first}. does not hold, then \ref{itm:second}. does. 
\end{proof}

The next lemma will be useful to give a lower bound on the number of vertices becoming strongly reachable after the addition of some pebbles to $S$.
\begin{lemma}
\label{kornyezet}
Let $D$ be a pebbling distribution on $G$. Let $S$ be a connected component of the subgraph which is induced by $\mathcal{U}(D)$. Consider $D'$ such that $D\leq D'$. If there is an $s\in V(S)$ such that $s$ is 2-reachable under $\Delta_{D,D'}$ and each vertex of $S$ is reachable under $D'$, then $N[s]\subseteq \mathcal{T}(D')$, furthermore $N[s]\subseteq \Delta \mathcal{T}_{D,D'}.$ 
\end{lemma}

\begin{proof}
We show that each neighbor of $s$ is strongly reachable under $D'$. Let $v$ be a neighbor of $s$, and $u$ be a neighbor of $v$. $s$ is 2-reachable under $D'$ hence $v$ is reachable. 
\par If $u$ is reachable under $D$ or it is a vertex of $S$, then it is reachable under $D'$. Else $u$ is in \mbox{$\mathcal{U}(D)\setminus V(S)$}. So $v$ separates two connected components in the induced subgraph by $\mathcal{U}(D)$, so $v$ is reachable under $D$. $s$ is 2-reachable under $\Delta_{D,D'}$, thus $v$ is also 2-reachable under $D'$ and $u$ is reachable. 
\par $s$ was not reachable under $D$ hence its neighbors were not strongly reachable under $D$. Therefore $N[s]\subseteq \Delta\mathcal{T}_{D,D'}$
\end{proof}

\begin{proof2}{Theorem \ref{improvedbound}}
Indirectly assume that there is a graph $G$ such that $\pis(G)>\frac{15n}{4(\delta+1)}$ and $\diam(G)>2$. This means that each solvable distribution has strengthening ratio below $\frac{4(\delta+1)}{15}$.

Let $D_0$ be a pebbling distribution which satisfies the properties of Lemma \ref{startdistribution}. Let $D$ be an expansion of $D_0$ such that the strengthening ratio of $D$ is at least $\frac{4(\delta+1)}{15}$ and subject to this requirement $|D|$ is maximal. According to our first assumption $D$ is not solvable. We will show that either $|D|$ is not maximal or $D$ is not an expansion of $D_0$. The first one is shown if we give a distribution $D'$ such that $D<D'$ and $\mathcal{E}_{D,D'}\geq \frac{4(\delta+1)}{15}$. 
 We will give $\Delta_{D,D'}$ instead of $D'$. Clearly $D$, and $\Delta_{D,D'}$ together are determine $D'$.
\par At each case we will assume that the conditions of the previous cases do not hold.
\par
\textbf{Case A:} There exist $u,v\in \mathcal{U}(D)$ such that $d(u,v)=3$ and there is a vertex $w$ on a shortest path between $u$ and $v$ which is contained in $\mathcal{H}(D)$.

W.l.o.g. assume that $w$ is a neighbor of $v$. Then let $\Delta_{D,D'}$ be the following: 
$$\Delta_{D,D'}(x)=\begin{cases}
4 & \text{if } x=u,\\
3 & \text{if } x=v,\\
0 & \text{otherwise}.
\end{cases}$$ 
\par 
$v$ is 4-reachable under $D'$, because $w$ is reachable under $D$ and it gets a pebble from $u$ under $\delta_{D,D'}$, so $w$ is 2-reachable without the three pebbles of $v$. This means that each vertex of the closed neighborhood of $u$ and $v$ are strongly reachable.
\par $N[u]$ and $N[v]$ are disjoint vertex sets and they are subsets of $\Delta \mathcal{T}(D,D')$. Hence
$$\mathcal{E}(D,D')\geq \frac{|N[u]\cup N[v]|}{|\Delta_{D,D'}|}\geq \frac{2(\delta+1)}{7} >\frac{4}{15}(\delta+1),$$
so $|D|$ was not maximal.

\textbf{Case B:} $\max_{u,v\in V(S)}d_B(u,v)\geq 4$

The conditions of case A are not satisfied, therefore there is a path in $S$ whose length is four in both $S$ and $B$ by Lemma \ref{diam4van}.
\par Apply Lemma \ref{diam4}. If there are vertices $u$ and $v$ from $V(S)$ such that $d_B(u,v)=2$ and $|N_B[u]\cup N_B[v]|$ $\geq \frac{28}{15}(\delta+1)$, then let $w$ be a common neighbor of $u$ and $v$ and choose $\Delta_{D,D'}$ as follows:
  $$\Delta_{D,D'}(x)=\begin{cases}
2 & \text{if } x\in\{u,v\},\\
3 & \text{if } x=w,\\
0 & \text{otherwise}.
\end{cases}$$ 
Each of $u,v,w$ is 4-reachable, hence:
\begin{equation*}
\Delta t_{D,D'}\geq |N[u]\cup N[v]|\geq \frac{28}{15}(\delta+1) ,
\end{equation*}
$|\Delta_{D,D'}|=7$, thus $\mathcal{E}(D,D')\geq \frac{4}{15}(\delta+1)$.
\par
 If there is no such an $u,v$ pair, then by Lemma \ref{diam4} there is a path $a,b,c,d,e$ in $S$ such that $d_B(a,e)=d_S(a,e)=4$, and $|N(b)\cap N(c)|\geq \frac{2}{15}(\delta+1)$. Consider $\Delta_{D,D'}$ as  follows:
  $$\Delta_{D,D'}(x)=\begin{cases}
4 & \text{if } x\in\{a,e\},\\
0 & \text{otherwise}.
\end{cases}$$ 
The vertices of set $N[a]\cup N[e]\cup (N[b]\cap N[d])$ are 2-reachable, thus they are also strongly reachable.
\begin{equation*}
\Delta t\geq |N[a]\cup N[e]\cup (N[b]\cap N[d])|=|N[a]|+|N[e]|+|N[b]\cap N[d]|\geq \frac{32}{15}(\delta+1), 
\end{equation*}
$$\mathcal{E}(D,D')\geq \frac{32(\delta+1)}{8\cdot 15}=\frac{4(\delta+1)}{15}.$$

\textbf{Case C:} 
$\max_{u,v\in V(S)}d_B(u,v)= 3$.

If the conditions of Case $A$ do not hold, then we can use Lemma \ref{osszefugg} because of Fact \ref{teny}. Let $(u,v)$ be the edge whose neighborhood size is at least $|\frac{4}{3}(\delta+1)|$. We will use this property only in the fourth subcase. 
\par Consider the set $\mathcal{K}$, which is a set of vertex sets.
$K$ is an element of $\mathcal{K}$ if and only if $K$ is a subset of $V(S)$ such that for all $k,j\in K$, $k\neq j$ implies that $d_B(k,j)\geq 3$, $|K|\geq 2$  and $K$ is maximal (we can not add an element to $K$). 
$\max_{u,v\in V(S)}d_B(u,v)= 3$ implies that $\mathcal{K}$ is not empty.
\par The objective in this case is to use Lemma \ref{kornyezet} for the vertices of $K$. Because this means that the vertices of $\cup_{k\in K}N[k]$ are strongly reachable. Furthermore, $N[k_1]$ and $N[k_2]$ are disjoint if $k_1,k_2\in K$ and $k_1\neq k_2$. These imply that $\Delta t \geq \cup_{k\in K}|N(k)|\geq |K|(\delta+1)$. To use this Lemma we need to give a proper $\Delta_{D,D'}$ distribution and check that each vertex of $S$ is reachable and each vertex of $K$ is 2-reachable under it. 
\par 
There are four subcases here:

\textbf{Subcase 1:} $\forall s\in V(S)$ $d_B(v,s)\leq 2$.

Let $K$ be an arbitrary element of $\mathcal{K}$. Note that $v\notin K$.
  $$\Delta_{D,D'}(x)=\begin{cases}
4 & \text{if } x=v,\\
1 & \text{if } x\in K,\\
0 & \text{otherwise}.
\end{cases}$$
Each vertex of $S$ is reachable with the pebbles placed at $v$ and the vertices of $K$ are 2-reachable.
$$\mathcal{E}(D,D')\geq \frac{|K|(\delta+1)}{4+|K|}\geq \frac{1}{3}(\delta+1).$$  

\textbf{Subcase 2:} $\forall s\in V(S)$ $\min(d_B(u,s),d_B(v,s))\leq 2$, but $\exists w\in V(S)\ d_B(v,w)=3$.
 
Choose $K$ such that $v\in K$. Such a $K$ is exists. 
  $$\Delta_{D,D'}(x)=\begin{cases}
3 & \text{if } x\in\{u,v\},\\
1 & \text{if } x\in K\setminus\{v\},\\
0 & \text{otherwise}.
\end{cases}$$ 
$u$ and $v$ are 4-reachable, hence all vertices of $S$ are reachable. Furthermore, each vertex of $K$ is 2-reachable.  
$$\mathcal{E}(D,D')\geq \frac{|K|(\delta+1)}{6+|K|-1}\geq \frac{2}{7}(\delta+1).$$ 

\textbf{Subcase 3:}
$\exists s\in V(S)\ d_B(s,u)=d_B(s,v)=3$ and $\{s,v\}\notin \mathcal{K}$.

$\{s,v\}$ is a subset of some elements of $\mathcal{K}$. Choose $K$ as one of these, $|K|\geq 3$.
$$\Delta_{D,D'}(x)=\begin{cases}
8 & \text{if } x=v,\\
1 & \text{if } x\in K\setminus\{v\},\\
0 & \text{otherwise}
\end{cases}$$ 
Each vertex of $S$ is reachable with the pebbles placed at $v$ and the vertices of $K$ are 2-reachable.  
$$\mathcal{E}(D,D')\geq \frac{|K|(\delta+1)}{8+|K|-1}\geq \frac{3}{10}(\delta+1)$$

\textbf{Subcase 4:} $\exists s\in V(S)\ d_B(s,u)=d_B(s,v)=3$ and $\{s,v\}\in \mathcal{K}$.

  $K=\{s,v\}$
	$$\Delta_{D,D'}(x)=\begin{cases}
4 & \text{if } x\in K,\\
0 & \text{otherwise}.
\end{cases}$$
$K=\{s,v\}$ means that each vertex of $S$ is in $N^2[s]\cup N^2[v]$, thus each vertex of $S$ is reachable. 
$N[s]\cap(N[u]\cup N[v])=\emptyset$ hence:
$$\Delta t_{D,D'}\geq |N[s]\cup N[v]\cup N[u]|=|N[s]|+|N[v]\cup N[u]|\geq \frac{7}{3}(\delta+1),$$
$$\mathcal{E}(D,D')\geq \frac{7}{24}(\delta+1).$$

\textbf{Case D:} $\max_{u,v\in V(S)}d_B(u,v)\leq 2$.

In this case if we put 4 pebbles to an arbitrary vertex $s$ of $S$, then all vertices of $S$ and $N[s]$ becomes strongly reachable.

\textbf{Subcase 1: $|V(S)|\geq \frac{16}{15}(\delta+1)$}.

Let $v$ be a vertex of $S$. 
$$\Delta_{D,D'}(x)=\begin{cases}
4 & \text{if } x=v,\\
0 & \text{otherwise}.
\end{cases}$$ 
$$\mathcal{E}(D,D')\geq \frac{16(\delta+1)}{15\cdot 4}=\frac{4}{15}(\delta+1)$$  

\textbf{Subcase 2a:} $\exists u,v\in V(s)$ such that $|N[u]\cup N[v]|\geq \frac{16}{15}(\delta+1)$ and $u$ and $v$ are neighbors.

$$\Delta_{D,D'}(x)=\begin{cases}
4 & \text{if } x=v,\\
0 & \text{otherwise}.
\end{cases}$$ 
Each vertex of $S$ is reachable and $u$ and $v$ are 2-reachable under $\Delta_{D,D'}$. Using Lemma \ref{kornyezet} we get that the neighborhoods of $u$ and $v$ are both strongly reachable.
$$\mathcal{E}(D,D')\geq \frac{|N[u]\cup N[v]|}{4}\geq\frac{16(\delta+1)}{15\cdot 4}=\frac{4}{15}(\delta+1).$$ 

\textbf{Subcase 2b:} $\exists u,v\in V(s)$ such that $|N[u]\cup N[v]|\geq \frac{16}{15}$ and $u$ and $v$ share a common neighbor $w\in V(s)$.

$$\Delta_{D,D'}(x)=\begin{cases}
4 & \text{if } x=w,\\
0 & \text{otherwise}.
\end{cases}$$ 

We can say the same like in the previous case.

\textbf{Subcase 3a:} $|V(S)|\leq \frac{14}{15}(\delta+1)$ and $\forall u,v\in V(S)\exists h\in \mathcal{H}(D)\ h\in N(u)\cap N(v)$.

Choose $v$ as an arbitrary vertex of $S$. 
$$\Delta_{D,D'}(x)=\begin{cases}
2 & \text{if } x=v,\\
0 & \text{otherwise}.
\end{cases}$$ 
If $s$ is a vertex of $S$ other than $v$, then there is $h\in \mathcal{H}(D)$ which is a neighbor of both of them. $h$ is reachable under $D$, under $D'$ we have two additional pebbles so we can move a pebble to $s$ from $v$ through $h$. Thus each vertex of $S$ is reachable under $D'$, so we can apply Lemma \ref{kornyezet} for vertex $v$.
 $$\mathcal{E}(D,D')\geq \frac{|N[v]|}{2}\geq\frac{\delta+1}{2}.$$ 
 
\textbf{Subcase 3b:} $|V(S)|\leq \frac{14}{15}(\delta+1)$ and $\exists u,v\in V(S)\nexists h\in\mathcal{H}(D) h\in N(u)\cap N(v)$.

The diameter of $S$ (with respect to the distance defined in $B$) guarantees that either $u$ and $v$ are neighbors or they share a common neighbor $w\in V(B)$. Furthermore, in this subcase $w\in V(S)$. $u$ has at least $\delta-(\frac{14}{15}(\delta+1)-1)=\frac{1}{15}(\delta+1)$ neighbors in $\mathcal{H}(D)$, but none of them is a neighbor of $v$. Hence $|N[u]\cup N[v]|\geq \frac{16}{15}(\delta+1)$. This is subcase 2a or 2b.

\textbf{Subcase 4:} $\exists v \in V(S)$, such that $\forall s\in V(S)\ d_B(v,s)=1$.

$$\Delta_{D,D'}(x)=\begin{cases}
2 & \text{if } x=v,\\
0 & \text{otherwise}.
\end{cases}$$ 
Each vertex of $S$ is reachable under $D'$ so we apply Lemma \ref{kornyezet} again and get that $\mathcal{E}(D,D')\geq \frac{\delta+1}{2}$.
\vspace{1cm}
\par
We have handled all cases when $|V(S)|\geq \frac{14}{15}(\delta+1)$ or $|V(S)|\leq \frac{16}{15}(\delta+1)$. So in the next sections we assume that  $\frac{14}{15}(\delta+1)< |V(S)|< \frac{16}{15}(\delta+1)$. 
Before we continue, we need one more definition. Let $\mathcal{S}$ be the set of connected components of the graph which is induced by $\mathcal{U}(D)$. Then we say that $S\in \mathcal{S}$ is \textit{isolated} in $\mathcal{S}$ if for any other $S'$ in $\mathcal{S}$  $d_G(S,S')\geq 3$.

\textbf{Subcase 5:} $\exists S \in \mathcal{S}$ such that $S$ is not isolated.

Exists $S'$ and $u\in S$, $v\in S$ such that $d(u,v)=2$. 
$$\Delta_{D,D'}(x)=\begin{cases}
4 & \text{if } x=u,\\
3 & \text{if } x=v,\\
0 & \text{otherwise}.
\end{cases}$$ 

$u$ and $v$ are both 4-reachable, hence all vertices of $S$ and $S'$ are reachable, furthermore they are strongly reachable.
  $$\mathcal{E}(D,D')\geq \frac{|V(S)\cup V(S')|}{7}\geq \frac{2\cdot\frac{14}{15}(\delta+1)}{7} =\frac{4(\delta+1)}{15}$$ 

\textbf{Subcase 6:} $S$ is isolated in $\mathcal{S}$ and $|N[S]|\geq \frac{16}{15}(\delta+1)$.

Let $s$ be an arbitrary vertex of $S$. Then:
$$\Delta_{D,D'}(x)=\begin{cases}
4 & \text{if } x=s\\
0 & \text{otherwise}
\end{cases}$$ 
Each vertex of $S$ becomes strongly reachable. We show that the same is true for any vertex of $N(S)$. Consider $h\in N(S)$. $h$ is not strongly reachable under $D$, but all of its non reachable neighbors under $D$ are contained in $S$, because $S$ is isolated. Thus under $D'$ $h$ is strongly reachable which gives the following result:
  $$\mathcal{E}(D,D')\geq \frac{|N[V(S)]|}{4}\geq \frac{\frac{16}{15}(\delta+1)}{4} =\frac{4(\delta+1)}{15}.$$ 

\textbf{Subcase 7:} $S$ is isolated in $\mathcal{S}$ and $\exists h \in \mathcal{H}(D)$ such that for each $s\in V(S)$ $d_G(h,s)\leq 2$.

$$\Delta_{D,D'}(x)=\begin{cases}
3 & \text{if } x=h,\\
0 & \text{otherwise}.
\end{cases}$$ 
Each vertex of $S$ becomes strongly reachable thus:
  $$\mathcal{E}(D,D')\geq \frac{|V(S)|}{3}>\frac{\frac{14}{15}(\delta+1)}{3}\geq\frac{4(\delta+1)}{15}.$$ 

\textbf{Subcase 8:} None of the previous cases hold.

In this case we will get a contradiction with $D_0\leq D$. We summarize what we know about $D$:
\begin{itemize}
\item $\forall S \in \mathcal{S}\ \max_{u,v\in V(S)}(d_B(u,v))= 2$,
\item $\forall S \in \mathcal{S}\ \frac{14}{15}(\delta+1)<|V(S)|<\frac{16}{15}(\delta+1)$,
\item $\forall S$ is isolated in $\mathcal{S}$,
\item $\forall S \in \mathcal{S}\ |N[S]|<\frac{16}{15}(\delta+1)$,
\item $\nexists h\in \mathcal{H}(D)$ such that the distance between $h$ and any vertex of $S$ is at most two, where $S\in \mathcal{S}$.

\end{itemize}
First of all, the diameter of $G$ is at least $3$, hence some pebbles have been placed, so $\mathcal{H}(D)$ is nonempty. Fix a component $S\in \mathcal{S}$. $\mathcal{H}(D)\cap N(S)$ is also nonempty, because $G$ is connected.  
Consider an $h\in \mathcal{H}(D)\cap N(S)$. The last property guarantees that there is a vertex $v$ in $V(S)$ such that $d(v,h)=3$. A neighbor of $h$ is in $V(S)$. Denote this vertex with $u$. The first property and  $d(v,h)=3$ together imply that $u\in N^2(v)\cap S$.

\par $h$ is in $N^3(v)$, hence $N[h]\cap N[v]=\emptyset$. $N[v]\subseteq N[S]$:
$$|N[h]\cap N[S]|\leq |N[S]\setminus N[v]|= |N[S]|-|N[v]|<\frac{16}{15}(\delta+1)-(\delta+1)=\frac{1}{15}(\delta+1).$$
So we have: $|N[h]\setminus N[S]|\geq \frac{14}{15}(\delta+1)$. $u$ is contained in $S$, hence all of its neighbors are contained in $N[S]$, thus: 
$$|N[u,h]|\geq |N[u]|+|N[h]\setminus N[S]|\geq \frac{29}{15}(\delta+1).$$
$u$ is in $S$, so it is not reachable under $D$, but $D$ is an expansion of $D_0$ where $u$ has to be reachable because $|N[u,h]|\geq \frac{29}{15}(\delta+1)$ and $(u,h)\in E(G)$. This is a contradiction.
\par
We have seen that in each case we have a contradiction, so our assumption was false, hence the theorem is true.
\end{proof2}

Using this theorem we can prove that the upper bound of Bunde \textit{et al.} can not hold with equality. 

\begin{claim}
There is no connected graph $G$ such that $\pis(G)=\frac{4n}{\delta+1}$.
\end{claim}

\begin{proof}
Theorem \ref{improvedbound} shows that the optimal pebbling number of graphs whose diameter is at least three is smaller. So we have to check only diameter two and complete graphs whose optimal pebbling number is at most 4. $\frac{4n}{\delta+1}\geq 4$ and equality holds only for the complete graph, but $\pis(K_n)=2$.
\end{proof}

\begin{cor}
For any connected graph $G$ $\pis(G)<\frac{4n}{\delta+1}$ and this bound is sharp.
\end{cor}

Muntz \textit{et al.} \cite{smalldiam} characterize diameter three graph graphs whose optimal pebbling number is eight. Their characterization can be reformulated in the following weird statement:

\begin{claim}
Let $G$ be a diameter $3$ graph. $\pis(G)=8$ if and only if there are no vertices $x,u,v$ and $w$ such that $N^2[x]\cup N[u]\cup N[v]\cup N[w]=V(G)$. 
\end{claim} 

Theorem \ref{improvedbound} can be used to establish a connection between this unusual domination property and the minimum degree of the graph. Note that this is just a minor improvement of the trivial $\frac{1}{2}n-1$ upper bound. 

\begin{cor}
Let $G$ be a diameter $3$ graph on $n$ vertices. If there are no vertices $x,u,v$ and $w$ such that $N^2[x]\cup N[u]\cup N[v]\cup N[w]=V(G)$, then the minimum degree of $G$ is at most $\frac{15}{32}n-1$.
\end{cor}

\section{Graphs with high girth and low optimal pebbling number}
\label{highgirth}

The authors of \cite{ladder} proved, for $k\ge 6$, that the family of connected graphs having $n$ vertices, minimum degree $k$, and girth at least $2t+1$ has $\pis(G)/n\rightarrow 0$ as $t\rightarrow\infty$.
They ask (Question 6.3) whether the same is true for  $k\in\{3,4,5\}$. 
We answer this affirmatively for $k\in\{4,5\}$.
The case $k=3$ remains open.

\begin{theorem}
Suppose $G$ is a connected graph of order $n$ with $\delta(G)\geq k$ and $\gir(G)\geq 2t+1$.
If $k\ge 4$ then $\lim_{t\rightarrow \infty} \pi^*(G)/n =0$.
\end{theorem}

\begin{proof}
For $k,t\in Z^+$ with $k\geq 3$, let $L= 1+ k\frac{(k-1)^t-1}{k-2}$.
Suppose $G$ is a connected graph of order $n$ with $\delta(G)\geq k$ and $\gir(G)\geq 2t+1$, and consider the following experiment consisting of two steps. 
In the first step, place $2^t$ pebbles on a vertex $v$ with probability $p$, independently for each $v$. 
In the second step, one pebble on every vertex not reachable by pebbles placed in the first step.
Clearly the pebbling distribution is solvable. 

Let $X$ be the expected number of pebbles used in the experiment. 
The probability that $v$ is not reachable by the pebbles places in the first step is at most $(1-p)^L$ because, if at least one vertex in the closed ball of radius $t$ around $v$ has $2^t$ pebbles on it, we can solve $v$.
Thus
$$E(X)\leq 2^tpn + (1-p)^Ln,$$
and so  $\pi^*(G)\leq (2^tp +(1-p)^L)n$.
In particular, for $p= \ln{(L/2^t)}/L$, using the bound $1+x\leq e^x$, we get
$$\pi^*(G) \leq (\ln(L/2^t)+1) \frac{2^tn}{L}.$$
We have $(k-1)^t< L< 3(k-1)^t$, and so $\pi^*(G) <(2+t\ln{(k-1)}) (\frac{2}{k-1})^t  \cdot n$.
\end{proof}

\section*{Acknowledgment}
The research of Andrzej Czygrinow is supported in part by Simons Foundation Grant \# 521777.
The research of Glenn Hurlbert is partially supported by Simons Foundation Grant 246436.
The research of Gyula Y. Katona is partially
supported by National Research, Development and Innovation Office
NKFIH, grant K116769. The research of Gyula Y. Katona and L\'aszl\'o F. Papp is partially
supported by National Research, Development and Innovation Office
NKFIH, grant K108947.

\end{document}